\documentclass{article}
\usepackage{amsmath,amssymb,amsfonts}
\usepackage{latexsym}
\usepackage{graphics}
\usepackage{mathptmx} 
\usepackage{amsthm}

\newtheorem{theorem}{Theorem}
\newtheorem{lemma}{Lemma}
\newtheorem{proposition}{Proposition}
\newtheorem{corollary}{Corollary}

\newtheorem{example}{Example}
\newtheorem{remark}{Remark}

\newcommand{\eps}{\varepsilon}

\newcommand{\m}{M}

\newcommand{\e}{{\rm e}}

\newcommand{\R}{{\mathbb R}}
\newcommand{\V}{{\mathcal{V}}_0}

\begin{document}

\title{Completeness of the trajectories of
particles coupled to a general force field}

\author{A.M. Candela, A. Romero and
M. S\'anchez}
\date{}

\maketitle

\begin{abstract}
We analyze the extendability of the solutions to a certain second
order differential equation on a Riemannian manifold $(M,g)$,
which is defined by a general class of forces (both prescribed on
$M$ or de\-pen\-ding on the velocity). The results include the
general time--dependent anholonomic case, and further refinements
for autonomous systems or forces derived from a potential are
obtained. These extend classical results for Lagrangian and
Hamiltonian systems. Several examples show the optimality of the
assumptions as well as the utility of the results, including an
application to relativistic pp--waves.
\end{abstract}




\tableofcontents

\section{Introduction}
{\em Completeness} is an essential property for the curves which
are extremal of some classical Lagrangian fields or, with more
generality, which solve the differential equation satisfied by the
trajectories of the particles accelerated by different types of
forces on a Riemannian manifold $(M,g)$. Its interest appears both
from the geometrical and the mechanical point of view. Recall that
the simplest case of the geodesics of $(M,g)$ allows to understand
better the global structure of the manifold. If one assumes that
$(M,g)$ is (geodesically) complete, the incompleteness of the
trajectories suggests that an infinite amount of energy will be
consumed by the accelerating forces. A priori, this is an
undesirable property for a mechanical system, and can be used to
disregard it as a physically realistic model. Nevertheless, in
some cases the equation of such trajectories may have nice
physical interpretations. For instance, sometimes the equation of
the accelerated trajectories may be equivalent to the equation of
the geodesics of a Lorentzian manifold (see, for example,
\cite{CFS}, \cite{RoSa}, \cite{SaGRG98} or the review \cite{CaSa}). So, the
incompleteness of such trajectories may be connected to the
lightlike or timelike incompleteness of some physically reasonable
spacetimes, and therefore, it can be related to the celebrated
relativistic theory of singularities. In this paper, we are
providing several general criteria which ensure the completeness of a
wide class of trajectories of accelerated particles in a
Riemannian setting, being its optimality discussed by means of
several examples. Essentially, this topic has remained dormant
since the results in the seventies. So, we use a simple approach
and language, which makes apparent the unsolved questions in that
epoch, and possible lines of future research are also pointed out.

\subsection{Setting}

More precisely, let $(\m,g)$ be a (connected, finite--dimensional)
Riemannian manifold and denote by $\pi : \m \times \R
\longrightarrow \m$ the natural projection. Giving a (1,1) smooth
tensor field $F$ along $\pi$ and a smooth vector field $X$ along
$\pi$, let us consider the second order differential equation
\[
\hspace*{2.9cm}\frac{D\dot\gamma}{dt}(t)\ =\ F_{(\gamma(t),t)}\
\dot\gamma(t) + X_{(\gamma(t),t)},\hspace*{2.9cm}\mathrm{(E)}
\]
where $D/dt$ denotes the covariant derivative along $\gamma$
induced by the Levi--Civita connection of $g$ and $\dot\gamma$
represents the velocity field along $\gamma$. Observe that (E)
describes the dynamics of a classical particle under the action of
a force field $F$, which linearly depends on its velocity, and an
external force field $X$, which is independent of the motion of
the particle. In the case when both $F$ and $X$ are time
independent, the previous equation reads
\[
\hspace*{3.55cm}\frac{D\dot\gamma}{dt}(t)\ =\ F_{\gamma(t)}
\dot\gamma(t) + X_{\gamma(t)}\hspace*{3.55cm}\mathrm{(E_0)}
\]
and is called the \textit{autonomous} equation, using
the term \textit{non--autonomous} equation if at least one between
$F$ and $X$ is time dependent.

Taking $p\in \m$ and $v\in T_p\m$, there exists a unique
inextensible smooth curve $\gamma : I \to \m$, $0\in I$, solution
of (E) which satisfies the initial conditions
\[
\gamma(0) = p,\quad \dot\gamma(0) = v.
\]
 Such a curve is called {\sl complete} if $I = \R$ and
{\sl forward} (resp. {\sl backward}) {\sl complete} when $I = (a,b)$ with
$b=+\infty$ (resp. with $a=-\infty$).  As far as we know, only the
(holonomic) case when $F=0$ and $X$ comes from the
gradient of a potential function $V$, has been systematically
studied in the literature (see \cite{AM}, \cite{Go}, \cite{WM}). Even more,
accurate results have been stated only for a
time--independent potential (see \cite[Chapter 3]{AM}), being the
results for the non--autonomous case rather vague (see \cite{Go}).
Our study will  cover all the previous cases,
specially the anholonomic and time--dependent ones.

\subsection{Interpretations}

For the interpretation of $F$, recall that it can be decomposed
as
\[
F\ =\ S+H,
\]
where $S$ is the self--adjoint part of $F$ with respect to
$g$, and $H$ the skew--adjoint one. On one hand, the
bound of the eigenvalues of $S$ (which may vary with $(p,t)\in
\m\times \R$) becomes natural to ensure that the $F$--forces will
not carry out an infinite work in a finite time. Frictional forces
are typically proportional to the velocity and opposed to it, so,
they can be described by means of some $S$ with non--positive
eigenvalues. On the other hand, magnetic fields may be
classically described by the skew--adjoint part $H$ (see \cite{LL}).

In this paper our approach differs from the previous ones in
\cite{AM}, \cite{Go} where Lagrangian or Hamiltonian techniques are used.
In fact, we focus directly  on the interpretation of the velocity
of each trajectory for equation (E) as an integral curve of a
certain vector field $G$ ({\em second order equation}) on
$TM\times \R$.
This is carried out first in the autonomous case (Section
\ref{s2}), where the vector field can be redefined just on $TM$,
extending the well--known geodesic vector field in Riemannian
geometry (or the Lagrangian vector field for regular Lagrangians).
In the non--autonomous case (Section \ref{s3}), we show how  the
results and techniques of the autonomous case can be adapted to
the vector field $G$ on $TM\times \R$. Even though this
possibility was pointed out by Gordon \cite{Go} in the framework
of Hamiltonian systems, our approach is  more direct and accurate.
In both cases (autonomous and
non--autonomous), we include a special study of the case when the
force vector field can be derived from a potential.

\subsection{Statement of the main results}

Some notions are needed to state our results. Put in the
time--in\-de\-pen\-dent case
\[
S_{\mathrm{sup}} := \sup_{\underset{\|v\|=1}{v\in T\m}} g(v,S v),
\quad S_{\mathrm{inf}} := \inf_{\underset{\|v\|=1}{v\in T\m}}
g(v,S v), \quad \| S \| : =
\max\{|S_{\mathrm{sup}}|,|S_{\mathrm{inf}}|\}.
\]
In the non--autonomous case, consider the analogous notions
$S_{\mathrm{sup}}(t)$, $S_{\mathrm{inf}}(t)$, $\| S(t) \|$
computed for each slice $M\times \{t\}$. We say that $S$ is {\em
bounded} (resp. {\em upper bounded}; {\em lower bounded}) {\em
along finite times} when, for each $T>0$ there exists a constant
$N_T$ such that for all $t\in [-T,T]$ we have
\begin{equation}\label{bf}
\|S(t)\|  < N_T \;\; (\hbox{resp.} \;
S_{\mathrm{sup}}(t)<N_T ;\; \; S_{\mathrm{inf}}(t)>- N_T).
\end{equation}
Moreover, let $d$ be the distance canonically associated to the
Riemannian metric $g$. We say that (the norm of) a vector field
$X$ on $M$ {\em grows at most linearly in $\m$} if there exist
some constants $A,C>0$ such that
\begin{equation}\label{bx}
|X|_{p}\ :=\ \sqrt{g(X_p,X_p)}\ \leq\ A \
d(p,p_0)  + C \quad \text{for all} \quad p\in \m,
\end{equation}
for some fixed $p_0\in M$. With more generality, in the
non--autonomous case we say that a vector field $X$ along $\pi$,
{\em grows at most linearly in $\m$ along finite times} if for
each $T>0$ there exist $p_0\in M$ and some constants $A_T, C_T>0$
such that
\begin{equation}\label{bx2}
|X|_{(p,t)}\ \leq\ A_T \ d(p,p_0) +  C_T \quad \text{for all} \quad
(p,t)\in \m\times [-T,T].
\end{equation}
Obviously, conditions \eqref{bx}, \eqref{bx2} are independent of
the chosen point $p_0$.

 Our main result can be stated as follows.

\begin{theorem}\label{A1}
Let $(\m,g)$ be a complete Riemannian manifold and consider a
$(1,1)$ tensor field $F$ and a vector field $X$ both
time--dependent and smooth.

If $X$ grows at most linearly in $\m$ along finite times and the
self--adjoint part $S$ of $F$ is bounded (resp. upper bounded;
lower bounded) along finite times, then each inextensible solution
of {\rm (E)} must be complete (resp. forward complete; backward
complete).

In particular, if $\m$ is compact then any inextensible solution of
{\rm(E)} is complete for any $X$ and $F$.
\end{theorem}

\begin{remark}
(1) For the comparison with previous results in the literature,
recall that in Theorem \ref{A1}: (i) the problem is
time--dependent, (ii) $X$ is not necessarily the gradient of a
potential, and (iii) forces which depend linearly on the
velocities are allowed. Interpretations for frictional and
magnetic forces (Remark \ref{rdisipative}) or applications to
relativistic pp--waves (Example \ref{eppwave}), stress its proper
applicability. The optimality of the hypotheses in Theorem \ref{A1} is discussed
along the paper (see especially Example \ref{exx} and Remark \ref{renuevo}).

\smallskip
\noindent (2) The technique will suggest that a  superlinear
growth of the vector field $X$ may not destroy completeness: it is
only relevant the growth of the component of $X$ in the radial
component along the outside direction. Moreover, even  though the
forces above are always either independent or pointwise
proportional to the velocity, our techniques seem adaptable to
study also frameworks of higher order.  These considerations could
be used to give extensions of Theorem \ref{A1}, which might
constitute a further line of research.
\end{remark}

The remainder of the results are obtained in the relevant
case that $X$ comes from the gradient of a potential, so that
they can be compared easily with those in previous references,
where Lagrangian or Hamiltonian systems were considered.
\vspace{3mm}

Again, we need some notions to describe our next results.

Let $V:\m\times\R \rightarrow \R$ be a smooth time--dependent
potential, and emphasize as $\nabla^{\m}V$ the gradient of the
function $p \in \m \mapsto V(p,t)\in \R$, for each fixed $t\in
\R$. A function $U: M\times \R\rightarrow \R$ {\em grows at most
quadratically along finite times} if for each $T>0$ there exist $p_0\in\m$
and some constants $A_T, C_T>0$ such that
\begin{equation}\label{quadratic}
U(p,t)\ \leq\ A_T \ d^2(p,p_0) +  C_T \quad \text{for all} \;
(p,t)\in \m\times [-T,T]
\end{equation}
(again, this property is independent of the chosen $p_0$). Our main result
is then:

\begin{theorem}\label{A02}
Let $(\m,g)$ be a complete Riemannian manifold, $F$ a smooth
time--de\-pen\-dent $(1,1)$ tensor field on $\m$ with self--adjoint
component $S$ and $V:\m\times\R \to\R$ a smooth time--de\-pen\-dent
potential.

Assume that $S$ is bounded (resp. upper bounded; lower bounded)
along finite times and $-V$ grows at most  quadratically along
finite times.

If also $|\partial V/\partial t|  :M\times \R\rightarrow \R$ (resp.
$\partial V/\partial t$; $- \partial V/\partial t$) grows at most
quadratically along finite times, then each inextensible solution
of
\[
\hspace*{3cm}\frac{D\dot\gamma}{dt}(t)\ =\ F_{(\gamma(t),t)}\
\dot\gamma(t) - \nabla^M V (\gamma(t),t)\hspace*{2cm}
\mathrm{(E^*)}
\]
must be complete (resp. forward complete; backward
complete).
\end{theorem}

\begin{remark} (1)  From Theorem \ref{A02}
one can reobtain the conclusion (stated with more generality in Theorem \ref{A1})
that the completeness of $\mathrm{(E^*)}$ holds whenever $\m$ is compact.

\smallskip \noindent (2) The proof of Theorem \ref{A02} allows us to sharpen its
conclusions (see Theorem \ref{A01}).  When particularized to
autonomous systems, Theorem \ref{A02} (and, in particular,
Corollary \ref{c02}) extends the results by Weinstein and Marsden
in \cite{WM} and in \cite[Theorem 3.7.15]{AM} (see also our
discussion in Remark \ref{renuevo}). Furthermore, in the
non--autonomous case, it generalizes widely the results by
Gordon\footnote{\label{foot} In relation with \cite{Go} recall:
(1) By simplicity, here we do not consider the case when $M$ is
not complete (except in the discussion below Proposition
\ref{G01}), as the alternative hypotheses on $V$ in \cite{Go} (to
be bounded from below and a proper map) become quite restrictive
under our approach. However, one could combine such a type of
hypotheses on, for example, an incomplete end of the Riemannian
manifold $M$, with a general behavior of $V$ (as in Theorem
\ref{A02}) on the complete part (see also \cite[Prop. 4.1.21]{AMR}
for the hypotheses of properness in infinite dimension). (2) The
bounds on the growth of $\nabla V$ obtained by using an isometric
embedding in an Euclidean space $\phi: M\rightarrow \R^N$ and its
Euclidean norm $\| \cdot\|$ are less sharp than those obtained by
using the intrinsic distance $d$, as $\| \phi(p)-\phi(q)\| \leq
d(p,q)$ for $p,q\in M$. (3) In relation with the existence of
proper functions and embeddings (see \cite{Go}, \cite{Go73}), from the
technical viewpoint it is also worth pointing out
 that the existence of one embedding
$\phi$ as above with a {\em closed} image, is characterized
nowadays by the completeness of $g$ (see \cite{Mu} and compare
with \cite{Go73}). } in \cite{Go}.

\smallskip

\noindent (3) The estimate of the decreasing of $V$  agrees with
Theorem \ref{A1}, in the sense that the norm of the gradient of
the function $-A_T d^2(p,p_0) -C_T$ (say, in the open dense subset
of $M$ where it is smooth) grows  linearly. However, Theorem
\ref{A02} is not a consequence of Theorem \ref{A1}, as $\nabla^M
V$ may grow superlinearly even when $V$ is bounded.

\smallskip

\noindent (4) The optimality of the growth of $-V$ (as the
optimality of the growth of $X$ and $F$ in Theorem \ref{A1}) is
checked by simple 1--dimensional examples. The bound for $\partial
V/\partial t$ is also very general, and a relevant application of
this case for pp--waves is also developed at the end of this paper.
However, it is not so clear if our bound for the growth of
$\partial V/\partial t$ can be improved. Starting at Remark
\ref{rgordon}, a discussion is carried out, including the
introduction of a different type of bound (see Proposition
\ref{G01}). This question (which has been omitted in the
literature, as far as we know) may deserve to be studied
specifically further.
\end{remark}

\section{The autonomous case}\label{s2}

Along this section,  $X$ and $F$ are regarded as tensor fields
just on the (connected) manifold $\m$, so that equation
(E) simplifies into $\mathrm{(E_0)}$. An argument which extends
the construction of the geodesic flow in Riemannian geometry,
easily shows that there exists a vector field $G^0$ on the tangent
bundle $T\m$ such that its integral curves are precisely the
velocities $s \mapsto \dot\gamma(s)$ of the curves $\gamma$ which
solve equation $\mathrm{(E_0)}$ (see the detailed discussion for
the non--autonomous case in the next section). Recall that an
integral curve $\rho$ of a vector field defined on some bounded
interval $[a,b)$, $b<+\infty$, can be extended to $b$ (as an
integral curve) if and only if there exists a sequence
$\{t_n\}_n$, $t_n\nearrow b$, such that $\{\rho(t_n)\}_n$
converges (see \cite[Lemma 1.56]{ON}). The following technical
result follows directly from this fact.

\begin{lemma}\label{extend}
Let $\gamma: [0,b) \to \m$ be a solution of equation
$\mathrm{(E_0)}$ with $0<b<+\infty$. The curve $\gamma$ can be
extended to $b$ as a solution of $\mathrm{(E_0)}$ if and only if
there exists a sequence $\{t_n\}_n \subset [0,b)$ such that $t_n
\to b^-$ and the sequence of velocities $\{\dot\gamma(t_n)\}_n$ is
convergent in $T\m$.
\end{lemma}

\begin{remark}
Here, we deal with finite--dimensional Riemannian manifolds. So,
in order to apply Lemma \ref{extend} to the completeness of
trajectories, the local compactness of $M$, and then of $T\m$,
will be used. On the contrary, to extend the results in the
present article to (infinite--dimensional) Hilbert manifolds, the
standard tools require to induce a (complete) Riemannian metric on
$TM$ (see \cite{Eb}, \cite[Supplement 9.1.C]{AMR}).
\end{remark}

Furthermore, as in what follows we are going to use more than once
a classical subsolution argument,  we recall it here for
completeness (e.g., see \cite[Lemma 1.1]{T}).

\begin{lemma}\label{subsol}
Taking $f \in C^1(\R^2,\R)$, 
let
$w=w(t)$ be a subsolution of the differential equation
\begin{equation}\label{sub1}
\dot u\ =\ f(t,u) \quad  \hbox{on} \quad  \hbox{$[t_0, T)$,}
\end{equation}
i.e., $\dot w < f(t,w)$ on $[t_0, T)$. Then for every solution
$u=u(t)$ of \eqref{sub1} such that $w(t_0) \le u(t_0)$ we have
\[
w(t) < u(t) \quad \hbox{for all} \quad  \hbox{$t \in (t_0, T)$.}
\]
\end{lemma}
\subsection{General result}
\label{dueuno}

With the notation introduced in the previous section, we
have
\begin{equation}\label{adj}
g(v,Fv) = g(v,Sv)\qquad \hbox{for all \, $v\in T\m .$}
\end{equation}
Now we state an accurate result which gives sufficient conditions
for the forward/back\-ward completeness of each inextensible
solution of equation $\mathrm{(E_0)}$.

\begin{theorem}\label{A}
Let $(\m,g)$ be a complete Riemannian manifold, $X$ a smooth
time--inde\-pen\-dent vector field and $S$ the self--adjoint component
of a smooth time--inde\-pen\-dent $(1,1)$ tensor field $F$ on $\m$.

 If $\gamma: I \to \m$ is an inextensible solution of
$\mathrm{(E_0)}$ satisfying the following assumptions:
\begin{itemize}
\item[$(a_1)$] There exists a constant $c_\gamma > 0$ such that
\[
g(\dot\gamma(t), S \dot\gamma(t)) \ \le\ c_\gamma\
g(\dot\gamma(t),\dot\gamma(t))\quad \hbox{for all \, $t\in I$,}
\]
\item[$(a_2)$] The pointwise norm $|X|$ of $X$ is at most linear
on $|\gamma|$, i.e., there exists $r_\gamma > 0$ such that
\[
|X|_{\gamma(t)} \ \le\
r_\gamma\, (1+ |\gamma(t)|\,) \quad \hbox{for all \, $t\in I$,}
\]
where $|\gamma(t)|:=d(\gamma(t),p_0)$ is the Riemannian distance
to some fixed point $p_0 \in \m$,
\end{itemize}
then, $\gamma$ must be forward complete.
\vspace{2mm}

Analogously, if the inextensible solution $\gamma$ satisfies
$(a_2)$ while condition $(a_1)$ is replaced by the assumption
\begin{itemize}
\item[$(a_1')$] there exists a constant $c_\gamma > 0$ such that
\[
-g(\dot\gamma(t),S \dot\gamma(t)) \ \le\ c_\gamma\,
g(\dot\gamma(t),\dot\gamma(t)) \quad \hbox{for all \, $t\in I$,}
\]
\end{itemize}
then $\gamma$ must be backward complete.
\end{theorem}

\begin{proof} Without loss of generality,
let $I= [0,b)$, $0<b<+\infty$, be the domain of a
for\-ward--inexten\-si\-ble solution $\gamma$ of $\mathrm{(E_0)}$, and
put $p_0=\gamma(0)$, so that $|\gamma(t)|=d(\gamma(t),\gamma(0))$.
Writing \begin{equation}\label{u} u(t) :=
g(\dot\gamma(t),\dot\gamma(t)),\end{equation} it is enough to
prove that a constant $k > 0$ exists such that
\begin{equation}\label{inequality3}
u(t)\ \le\ k \quad\hbox{for all \, $t \in [0,b)$.}
\end{equation}
In fact, \eqref{inequality3} implies that $\dot\gamma(I)$  is
bounded in $TM$ and, being $(\m,g)$ complete, Lemma
\ref{extend} is applicable because of the compactness of the
bounded metric balls in $\m$. Hence, $\gamma$ can be extended to
$b$ in contradiction with its maximality assumption.

With the aim to prove \eqref{inequality3}, for any $t \in [0,b)$ equation
$\mathrm{(E_0)}$ implies
\[\begin{split}
\dot u(t)\ &=\ 2 g(\dot\gamma(t),F \dot\gamma(t)) +
2 g(\dot\gamma(t),X_{\gamma(t)}) \\
 &=\ 2 g(\dot\gamma(t),S \dot\gamma(t)) +
2 g(\dot\gamma(t),X_{\gamma(t)}), \end{split}
\]
and therefore, assumptions $(a_1)$ and $(a_2)$ give
\begin{eqnarray}\nonumber
\dot u(t)\ &\le&\ 2 c_\gamma g(\dot\gamma(t),\dot\gamma(t)) + 2
r_\gamma (1+|\gamma(t)|)
\sqrt{g(\dot\gamma(t),\dot\gamma(t))}\\
&\le&\ 2 c_\gamma g(\dot\gamma(t),\dot\gamma(t)) + r_\gamma^2
(1+|\gamma(t)|)^2+ g(\dot\gamma(t),\dot\gamma(t))\nonumber \\
& \le& \ (2 c_\gamma + 1) u(t) + 2 r_\gamma^2 + 2
r_\gamma^2|\gamma(t)|^2. \label{esplit}
\end{eqnarray}
 Thus, taking into account that
\[
|\gamma(t)|^2\ \leq\ \left(\int_0^t \sqrt{u(s)} ds\right)^2\ \leq\
b \int_0^t u(s) ds,
\]
and putting $v(t)= \displaystyle\int_0^t u(s) ds$, inequality
\eqref{esplit} yields
\begin{equation}\label{eineq}
\ddot v < k_1\, \dot v + k_2\, v + k_3
\end{equation}
for some constants $k_1, k_2, k_3 > 0$. Recall that any solution
$v_0$ of the linear ordinary differential equation obtained by
replacing the inequality in \eqref{eineq} by equality, is a $C^2$
function defined on all $\R$. Now, choosing $v_0$ such that
$v_0(0)= v(0) =0$, $\dot v_0(0)=u(0)$, applying twice Lemma
\ref{subsol} we have $v<v_0$, $\dot v <\dot v_0$ on all $(0,b)$
and, so, \eqref{inequality3} holds with $k= \max \{\dot v_0(t):
t\in [0,b]\}$.

Vice versa, let $\gamma : (-b,0] \to \m$ be a
backward--inextensible solution of $\mathrm{(E_0)}$. Then, the
reparametrization $\gamma^* : t\in[0,b) \mapsto \gamma^*(t) =
\gamma(-t) \in \m$ is a forward--inextensible solution of
\[
\frac{D\dot\gamma^*}{dt}(t)\ =\ (-F)_{\gamma^*(t)} \dot\gamma^*(t)
+ X_{\gamma^*(t)} \qquad\hbox{in \, $[0,b)$.}
\]
From $(a_1')$ it follows that $-F$ satisfies $(a_1)$ along
$\gamma^*$; whence $\gamma^*$ must be forward complete, that is,
$\gamma$ is backward complete.
\end{proof}

\begin{remark} \label{rdisipative} (1) Assumption $(a_1)$ in
the previous result means that $\lambda^+(t)$, the big\-gest eigen\-value
 of the operator $S$  along $\gamma$, is upper
bounded for $t\in I$. Notice that
 the speed of $\gamma$ increases maximally  when
each $\dot\gamma(t)$ lies in the $\lambda^+(t)$--eigen\-space. The
upper boundedness of $\lambda^+(t)$ implies that, even though its
speed may increase linearly, the curve $\gamma$ cannot cover an
infinite length in a finite time and, so, the trajectory is
forward complete. Recall that only upper boundedness is relevant
here. In fact, frictional forces are proportional to the velocity
(at least as a first approximation) and opposed to it (i.e., with
negative eigenvalues). So, they make the speed  to decrease and,
thus, the obtained results agree with the expectation that the
trajectory is defined for arbitrarily big times.

\smallskip

\noindent (2) Theorem \ref{A} also implies that, on any complete
Riemannian manifold $(\m,g)$, each inextensible solution of
equation $\mathrm{(E_0)}$ must be complete if $F$ is assumed to be
skew--adjoint and $|X|$ is bounded. Such a result applies to
magnetic fields both, in a non--relativistic setting and in the
relativistic one (when a suitable static vector field exists). In
fact,  it extends widely \cite[Corollary 2.4]{BCFR}, which was
stated for pure magnetics fields (i.e., $F$ is skew--adjoint and
$X\equiv 0$) on $(\m,g)$. Recall that any magnetic trajectory
$\gamma$ satisfies the conservation law $g(\dot \gamma (t), \dot
\gamma (t)) =\ $constant (depending on $\gamma$). This  crucial
property is used to get \cite[Corollary 2.4]{BCFR}, but it does
not hold for electric forces or other forces allowed by
$\mathrm{(E_0)}$.
\end{remark}

A direct consequence of Theorem \ref{A} is the following
result.

\begin{corollary}
Let $(\m,g)$ be a complete Riemannian manifold, $X$ a smooth
time--in\-de\-pen\-dent vector field and $S$ the self--adjoint component
of a smooth time--in\-de\-pen\-dent $(1,1)$ tensor field $F$ on $\m$.
If $X$ grows at most
linearly in $\m$ and $\|S\|$ is bounded, then all the inextensible solutions of ${\rm (E_0)}$ are
complete. In particular, this result holds if $\m$ is compact.
\end{corollary}

\begin{example}\label{exx} {\rm The optimal character of the bounds
in Theorem \ref{A} can be checked just taking $\m = \R$, $g
=dx^2$.
\\[0.5mm]
(1) {\em Optimality of the bound for $|X|$.} Put $F\equiv 0$ and
$X(x) = \mu_\eps (x) \frac{d}{dx}$ where $\mu_\eps(x)= (1+\eps)\,
x^{1+2\eps}$ for all $x\ge 1$ and some prescribed $\eps >0$. Thus,
in the region where $x(t)\geq 1$ equation $\mathrm{(E_0)}$ reduces
to
\begin{equation}\label{Ex1}
\ddot{x}(t) \ =\ (1+\eps) x^{1+2\eps}(t).
\end{equation}
Multiplying by $\dot x$, integrating with respect to $t$ and
considering the initial data $x(0)=1, \dot x(0)=0$, the
solution of equation \eqref{Ex1} satisfies
\[
\dot{x}^2(t) \ =\ x^{2(1+\eps)}(t) - 1.
\]
The solution of this Cauchy problem  is the inverse of
\begin{equation}\label{et}
t(x)\ =\ \int_1^x \frac{d\sigma}{\sqrt{\sigma^{2(1+\eps)} -1}},
\end{equation}
 which is defined for $t\in[0,b)$, with the maximum
$b$ equal to $\displaystyle\lim_{x\rightarrow +\infty}t(x)$ in \eqref{et}.
So, \eqref{Ex1} is incomplete whenever the power of the growth of
$|X|$ becomes bigger than the permitted linear one.\\[0.5mm]
(2) {\em Optimality of the bound for $\| F\|$.} For some $\eps>0$,
put $F \frac{d}{dx} =  \nu_\eps (x) \frac{d}{dx}$ with
$\nu_\eps(x)= (1+\eps) x^\eps$ for $x\ge 1$ and $X \equiv 0$.
Equation $\mathrm{(E_0)}$ reduces to
\begin{equation}\label{Ex2}
\ddot{x}(t) \ =\  (1+\eps)\, x^\eps(t)\ \dot{x}(t)
\end{equation}
whenever $x(t)\ge 1$. Thus, integrating with respect to $t$ and
considering the initial data $x(0)=\dot x(0)=1$, the solution
of equation \eqref{Ex2} satisfies
\[
\dot{x}(t) \ =\  x^{1+\eps}(t);
\]
whence, the solution of this Cauchy problem in this region is the
inverse of
\[
t(x)\ =\ \int_1^x  \frac{d\sigma}{\sigma^{1+\eps}}
\]
which shows the incompleteness of $x(t)$, as $\displaystyle\lim_{x\rightarrow
+\infty}t(x) < +\infty$. That is, incompleteness appears 
when $\|F\|$ is not bounded, even under slow
power growth.\\[0.5mm]
(3) {\em Role of the bound on $\|F\|$ for forward/backward
completeness.} Put now  $X \equiv 0$ and $F \frac{d}{dx} = - \mu
(x) \frac{d}{dx}$ where $\mu$ is defined on all $\R$ and
$\mu(x)=|x|$ for $|x|>1$. Equation $\mathrm{(E_0)}$ reduces to
\begin{equation}\label{Ex3}
\ddot{x}(t) \ =\ - |x(t)|\, \dot{x}(t)\qquad \hbox{whenever \,
$|x(t)|>1$.}
\end{equation}
As $F$ is self--adjoint and satisfies the hypothesis $(a_1)$ of
Theorem \ref{A}, then the solutions of \eqref{Ex3} are forward
complete. However, $x(t) = \frac{2}{t+1}$, $t \in (-1,0]$, yields
a backward inextensible and incomplete solution of \eqref{Ex3}.
Let us remark that $F$ may represent a  frictional force
(increasing with $|x|$ in an inhomogeneous medium), and the
backward incompleteness implies the divergence (with the lapse of
time) of the energy necessary to overcome such a force.}
\end{example}

\subsection{Trajectories under an autonomous potential}\label{s22}

Throughout this section we assume $X = -\nabla V$. If $F \equiv
0$, the completeness of inextensible solutions of equation
$\mathrm{(E_0)}$ has been studied in \cite[Theorem 2.1{\sl
(ii)}]{Go} if $V$ is bounded from below while in \cite{WM} if $V$
is unbounded from below (see also \cite[Theorem 3.7.15]{AM}).
Here, we generalize such results by including also the action of a
(1,1) tensor field $F$. In order to investigate the completeness
of equation $\mathrm{(E_0)}$, let us recall the following
comparison result (see \cite[Example 3.2.H]{AM} or \cite{CRS}).

\begin{lemma}[{\bf Comparison Lemma}]\label{comparison}
Let $\varphi :[0,+\infty) \to \R$ be a locally Lipschitz monotone
increasing function such that
\[
\varphi(s) > 0\quad \hbox{for all \quad $s\ge 0$}
\quad\hbox{and}\quad \int_0^{+\infty} \frac{ds}{\varphi(s)} =
+\infty.
\]
If an inextensible $C^1$ function $f=f(t)$ is such that
\[
\frac{df}{dt}(t)\ =\ \varphi(f(t))\quad \hbox{with $f(0) \ge 0$,}
\]
then it is defined for all $t \ge 0$.
\vspace{1mm}

Furthermore, if $h :[0,b) \to \R$ is a continuous function such
that $h(t) \ge 0$ for all $t \in [0,b)$ and
\[
\left\{\begin{array}{ll} \displaystyle h(t) \le h(0) + \int_0^t
\varphi(h(s)) d s  &\hbox{for all \;
$t \in [0,b)$,}\\[1mm]
h(0) \le f(0),&
\end{array}\right.
\]
then $h(t) \le f(t)$ for all $t \in [0,b)$.
\end{lemma}

According to \cite{WM} (see also \cite[Definition 3.7.14]{AM}) a
function $\V :[0,+\infty) \to \R$ is called {\em positively
complete}, if it is $C^1$, non--increasing and satisfies
\begin{equation}\label{epositivelycomplete}
\int_0^{+\infty} \frac {ds}{\sqrt{\alpha - \V(s)}}\ = +\infty,
\end{equation}
where $\alpha$ is a constant such that $\alpha > \V(0)$, hence
$\alpha > \V(s)$ for all $s \in [0,+\infty)$. It is easy to see
that this condition is independent of which $\alpha$ is chosen.

\begin{remark} \label{renuevo}
(1) If $\V$ is a positively complete function and  $\widetilde \V$
is a non--increasing $C^1$--function such that $\widetilde
\V(0)=\V(0)$ and $\widetilde \V\geq \V$ then $\widetilde \V$ is
also positively complete. So, a positively complete function $\V$
will be interesting when $\displaystyle\lim_{s\rightarrow
+\infty}\V(s)$ goes to $-\infty$ as fast as possible. Therefore, the
relevant limit equivalent to \eqref{epositivelycomplete} for any
non--increasing function on $[0,+\infty)$ is
 $\int_{s_0}^{+\infty} ds/\sqrt{-\V(s)}=+\infty,$
where $s_0$ is any point with $\V(s_0)<0$.

\smallskip
\noindent (2) In particular, the function
\[
\V(s)\ =\ - R_0\ s^2 \quad \hbox{where \,
$R_0 >0$,}
\]
is positively complete. Thus, so are, for example, $ \V(s)\ =\ -
s^\beta \ \log^\alpha(1+s) $ for any $\beta \in [0,2)$ and any
$\alpha>0$, as they decrease less fast than $-s^2$.
Anyway, choosing $\beta= 2$ and $0<\alpha\leq2$ or, with more generality, if $\V(s)$ is 
a $C^1$ non--increasing function on $[0,+\infty)$ such that
\[
\V(s)\ \approx\ - s^{2} (\log s)^2\dots (\underbrace{\log\log\dots\log}_{k - \hbox{times}} s)^2
\quad \hbox{if $s \to +\infty$,}
\]
for any $k \ge 1$, one also finds positively complete functions 
which decrease (slowly) faster than quadratically. 
Consistently with references \cite{AM}, \cite{AMR}, 
our results here will be stated by using positive completeness. With more generality, one could 
replace the hypothesis such as the at most linear behavior in Subsection \ref{dueuno} by a more general 
(but technical) assumption -- nevertheless, we prefer not to do this for the sake of simplicity.

\smallskip
\noindent (3) Notice that  functions such as $\V(s)= - R_0\
s^\beta$ are not positively complete when $\beta>2$. This agrees
with the fact that, for such a type of functions, $|\nabla \V|$ is
at most linear if and only if $\beta\leq 2$. So, even though our
overall linear bound for $|X|$ was the optimal power growth, some further results
will be obtained next.
\end{remark}

\begin{theorem}\label{A01}
Let $(\m,g)$ be a complete Riemannian manifold, $V$ a smooth
potential on $\m$ and $S$ the self--adjoint component of a smooth
time--in\-de\-pen\-dent $(1,1)$ tensor field $F$.
Let $\gamma: I \to \m$ be an inextensible solution of
\[
\hspace*{3.4cm}\frac{D\dot\gamma}{dt}(t)\ =\ F_{\gamma(t)}\
\dot\gamma(t) - \nabla
V_{\gamma(t)}.\hspace*{3.4cm}\mathrm{(E_0^*)}
\]
If $S$ satisfies condition $(a_1)$ (resp. $(a_1')$) in Theorem
\ref{A} and
\begin{itemize}
\item[$(a_3)$] a positively complete function $\V$ exists such that
\[
V(\gamma(t))\ \ge\ \V(d(\gamma(t),p_0)) \quad \mbox{for all \quad
$t \in I$,}
\]
for some $p_0\in M$,
\end{itemize}
then $\gamma$ is forward (resp. backward) complete.
\end{theorem}

\begin{proof}
As in the proof of Theorem \ref{A}, let $I= [0,b)$ and, by
contradiction, assume $0<b<+\infty$. Introducing again the squared norm
$u(t)$ (see \eqref{u}), it is enough to prove that inequality
\eqref{inequality3} holds for some constant $k > 0$ (boundedness
of $u$).

On one hand, by equations $\mathrm{(E_0^*)}$, \eqref{adj} and assumption $(a_1)$, for all
 $s\in[0,b)$ we have
\begin{equation}\label{dis1}
\frac{d}{ds}\Big(\frac{1}{2}\ u(s) + V\circ\gamma(s)\Big) \ =\
g(\dot\gamma(s),S\dot\gamma(s))\ \leq\ c_\gamma u(s).
\end{equation}
Hence, taking
\[
v(t)=\int_0^t u(s) ds, \quad \text{and thus} \quad \dot v(t)
=u(t),\; v(0) = 0,
\]
and integrating \eqref{dis1} on $[0,t]$, $t\in[0,b)$, we have
\begin{equation}\label{eq1}
\dot v(t) - 2 c_\gamma v(t)\, \le\ 2 (\alpha_\gamma -
V(\gamma(t))) \quad \text{for all}\quad t\in[0,b),
\end{equation}
with $\alpha_\gamma = \frac 12\ u(0) +
V(\gamma(0))$. On the other hand, we get
\begin{equation}\label{uno}
d(\gamma(t),p_0)\ \le\ d(\gamma(0),p_0) + d(\gamma(t),\gamma(0))\
\le\ l_\gamma(t) \quad\hbox{for all $t \in [0,b)$,}
\end{equation}
where we have put
\begin{equation}\label{tre}
l_\gamma(t)\ =\ d(\gamma(0),p_0) + \int_0^t
\sqrt{g(\dot\gamma(s),\dot\gamma(s))} ds.
\end{equation}
Thus, from \eqref{eq1}, assumption $(a_3)$ and \eqref{uno}, for all $t\in[0,b)$ we
obtain
\[
\dot v(t) - 2 c_\gamma v(t)\ \le\ 2 (\alpha_\gamma -
\V(d(\gamma(t),p_0))) \ \le\ 2 (\alpha_\gamma - \V(l_\gamma(t))),
\]
the last inequality as $\V$ is non--increasing. This property
also assures that, taking $\alpha > \max\{\alpha_\gamma,\V(0)\}$,  the function $v=v(t)$  satisfies
\begin{equation}\label{eq2}
\dot v(t) - 2 c_\gamma v(t)\ <\ 2 (\alpha - \V(l_\gamma(t)))
\qquad \hbox{for all $t\in[0,b)$,}
\end{equation}
and the right--hand side of the inequality is positive.

\vspace{2mm}

Now, let $v_0= v_0(t)$ be the solution of the associated equality
\begin{equation}\label{eq3}
\dot v_0(t) - 2 c_\gamma v_0(t)\ =\ 2 (\alpha - \V(l_\gamma(t)))
\end{equation}
with initial condition $v_0(0) = 0$; explicitly:
\begin{equation}\label{eq4}
v_0(t)\ =\ 2 \e^{2 c_\gamma t}\ \int_0^t \e^{-2 c_\gamma s}
\left(\alpha - \V(l_\gamma(s))\right) ds, \quad t\in[0,b).
\end{equation}
From \eqref{eq2} it follows that $v=v(t)$ is a subsolution of
\eqref{eq3} with $v(0) = v_0(0)$. Thus, applying Lemma
\ref{subsol} as before,
\[
v(t) < v_0(t) \quad\hbox{and}\quad u(t) = \dot v(t) < \dot v_0(t)
\;\; \quad \hbox{for all \; $t\in (0,b)$.}
\]
So, in order to estimate $\dot v_0(t)$, let us remark that
\eqref{eq4} implies
\[
\dot v_0(t)\ =\ 4 c_\gamma \e^{2 c_\gamma t}\ \int_0^t \e^{-2
c_\gamma s} \left(\alpha - \V(l_\gamma(s))\right) ds\ +\ 2
\left(\alpha - \V(l_\gamma(t))\right), \; t\in[0,b).
\]
Choosing any $t \in [0,b)$, the  properties of $\V$ imply $
\V(l_\gamma(s))\ \ge\ \V(l_\gamma(t))$ for all $s\in [0,t)$ and,
hence,
\[
\int_0^t \e^{-2 c_\gamma s} \left(\alpha - \V(l_\gamma(s))\right)
ds\ \le\ \frac{1}{2 c_\gamma} \left(1 - \e^{-2 c_\gamma t}\right)  \left(\alpha - \V(l_\gamma(t))\right).
\]
This implies
\[\begin{split}
\dot v_0(t)\ &\le\ 2 \e^{2 c_\gamma t}
\left(1 - \e^{-2 c_\gamma t}\right) \left(\alpha - \V(l_\gamma(t))\right) + 2 \left(\alpha - \V(l_\gamma(t))\right)\\
& \le\ k_\gamma \left(\alpha - \V(l_\gamma(t))\right)
\end{split}
\]
with $k_\gamma = 2 \e^{2 c_\gamma b}$ (note that $\alpha - \V$ is positive); whence,
\begin{equation}\label{eq5}
\sqrt{u(t)}\ \le \ \sqrt{k_\gamma \left(\alpha -
\V(l_\gamma(t))\right)} \qquad \hbox{for all \; $t\in[0,b)$.}
\end{equation}
On the other hand, from definition \eqref{tre} we have
\[
\frac{d l_\gamma}{dt}(t)\ =\ \sqrt{u(t)}\quad \hbox{and}\quad
l_\gamma(0)\ =\ d(\gamma(0),p_0).
\]
Thus, defining $\varphi(w) = \sqrt{k_\gamma \left(\alpha -
\V(w)\right)}$, from \eqref{eq5} it follows
\[
\frac{d l_\gamma}{dt}(t)\ \le \ \varphi(l_\gamma(t)) \qquad
\hbox{for all $t\in[0,b)$,}
\]
which implies
\begin{equation}\label{eq6}
l_\gamma(t)\ \le \ l_\gamma(0) + \int_0^t\varphi(l_\gamma(s)) ds
\qquad \hbox{for all $t\in[0,b)$.}
\end{equation}
Now, let $f=f(t)$ be the unique inextensible solution of the
Cauchy problem
\[
\left\{
\begin{array}{l}
\displaystyle \frac{df}{dt}(t)\ =\
 \varphi(f(t)),\\
f(0) = d(\gamma(0),p_0)\ (=l_\gamma(0) \ge 0).
\end{array}\right.
\]
As $\V$ is positively complete, a direct check of the properties of $\varphi$
implies that the first part of Lemma \ref{comparison} applies, so $f$ is defined for
all $t\geq 0$. Moreover, from \eqref{eq6} and the second part of Lemma \ref{comparison}
it follows $l_\gamma(t) \le f(t)$ for all $t \in [0,b)$. Thus, $l_\gamma(t)$
is bounded in $[0,b)$ and from \eqref{eq5} so is $u(t)$, as
required --the contradiction that $\gamma$ is extensible beyond
$b$ follows.

The case when $\gamma : (-b,0] \to \m$ is backward incomplete
follows analogously (as at the end of the proof of Theorem \ref{A}).
\end{proof}

The autonomous version of Theorem \ref{A02} is now a
straightforward consequence of the previous theorem and the
positive completeness of the function $\V (s)=-R_0s^2$ discussed
in Remark \ref{renuevo}. Concretely:

\begin{corollary}\label{c02} Let $(\m,g)$ be a complete
Riemannian manifold, $S$ the self--adjoint component of a smooth
time--independent $(1,1)$ tensor field $F$ on $\m$ and $V:\m
\to\R$ a smooth time--independent potential.  Assume that
$\|S\|<+\infty$ and $-V$ grows at most quadratically (i.e.,
$-V(p)\ \leq\ A\ d^2(p,p_0) +\ C$ in agrement with
\eqref{quadratic}).

Then each inextensible solution of $\mathrm{(E_0^*)}$
must be complete.  In particular, completeness of inextensible solutions of this
equation holds whenever $\m$ is compact.
\end{corollary}

\section{The non--autonomous case}\label{s3}

Next, the  non--autonomous case (E) will be reduced to the
autonomous one by working on the manifold $M\times \R$ (compare
with the classical approach in \cite[pp. 121--124]{CM}, for
instance), and the two main theorems in the Introduction will be
proven. Again, we consider first the general case. An analogous
reasoning also proves the case when the external force comes from
a potential, which is then widely analyzed.

\subsection{General result}

By taking into account the standard results on existence and
uniqueness of solutions to second order differential equations,
for each $(v_p,t_{0})\in TM\times \R$ ($p\in \m$, $v_p \in T_p\m$)
we can consider the unique inextensible solution
$\gamma_{_{(v_p,t_{_0})}}$ of (E) which satisfies
$\gamma_{_{(v_p,t_{_0})}}(t_0)= p$ and
$\dot\gamma_{_{(v_p,t_{_0})}}(t_0)=v_p$.

\begin{lemma}\label{vector_field}
There exists a unique vector field $G$ on $TM\times \R$ such that
the curves $t\mapsto \big(\dot\gamma_{_{(v_p,t_{_0})}}(t),t\big)$
are the integral curves of $G$.
\end{lemma}

\begin{proof}
Obviously, if such a $G$ exists, then it must be defined as
\[
G_{(v_p,t_{_0})}\ =\ \left.
\frac{d^2}{dt^2}\right|_{_{t=t_{_0}}}\gamma_{_{(v_p,t_{_0})}}(t)
+\left. \frac{\partial}{\partial t}\right|_{_{(v_p,t_{_0})}}.
\]
To check that its integral curves satisfy the required property,
let us consider
\[
\begin{array}{rccc}
\Psi: {\cal D} \subset & \R\times (TM\times \R) &\rightarrow &
TM\times \R \\ & (s,(v_p,t_0)) & \mapsto &
\big(\dot\gamma_{_{(v_p,t_0)}}(t_0+s), t_0+s\big)
\end{array}
\]
where ${\cal D}$ is the maximal domain of definition of $\Psi$ in
$\R\times (TM\times \R)$ ---recall that ${\cal D}\cap (\R\times
\{(v_p,t_0)\})$ is always an open interval, which contains $0$
(multiplied by $\{(v_p,t_0)\}$). Clearly, $\Psi$ defines a local
action on $TM\times \R$ (namely, $\Psi_{s+t}= \Psi_{s}\circ
\Psi_{t}$ whenever it makes sense) and the result follows.
\end{proof}

\begin{remark} Alternatively to the previous lemma, the vector
field $G$ may be defined locally as follows: let $(U;x^1,...,x^n)$
be a coordinate neighborhood on $M$ and consider the natural
coordinates $(x^1,...,x^n,\dot x^1,...,\dot x^n)$ on
$\pi_M^{-1}(U)$, where $\pi_M$ is the projection from $TM$ onto
$M$, i.e., for any $p\in U, v_p\in T_pM$: $x^i(v_p)\equiv x^i(p)$,
$\dot x^i(v_p)\equiv dx_p^i(v_p)$, $1 \leq i \leq n$. On
$\pi_M^{-1}(U) \times \R$, we have
\[
\begin{split}
G_{(v_p,t)}\ =\ &\sum_{i=1}^{n}\dot x^i(v_p) \left.
\frac{\partial}{\partial x^i}\right|_{v_p}
-\sum_{i=1}^{n}\left(\sum_{j,k=1}^n\Gamma_{j,k}^{i}(p)\dot
x^j(v_p)\dot x^k(v_p)\right) \left. \frac{\partial}{\partial \dot
x^i}\right|_{v_p}\\
&
+\sum_{i=1}^{n}\left(X^i(p,t)+\sum_{j=1}^{n}\dot
x^j(v_p)F_j^i(p,t)\right) \left. \frac{\partial}{\partial \dot
x^i}\right|_{v_p}
+ \left.  \frac{\partial}{\partial t}\right|_{t},
\end{split}
\]
where $\Gamma_{j,k}^{i}$ are the corresponding Christoffel
symbols, $X^i(p,t)=dx_p^i(X_{(p,t)})$ and\\
$F_j^i(p,t)= dx_p^i\left(F_{(p,t)}\,\left.\frac{\partial}{\partial
x^j}\right|_p\right)$.
\end{remark}

Now, in order to be able to use the previous results for the
autonomous case, consider the trajectories of equation ({\rm E})
as ``some'' of the integral curves of a vector field $\tilde G$ on
the tangent bundle $T(M\times \R)$. First, note that the
time--dependent tensor fields $X$ and $F$ on $M$ naturally induce
tensor fields $\tilde X, \tilde F$ on $\tilde M:= M\times \R$,
namely:
\begin{equation}\label{def1}
\begin{array}{c}
\tilde X_{(p,t_{_0})}= (X_{(p,t_{_0})},0)\equiv
X_{(p,t_0)}, \\[3mm]
\tilde
F_{\left(p,t_{_0}\right)}\left(v_p,\,s\,\frac{d}{dt}\mid_{_{t_{_0}}}\right)=\left(F_{(p,t_{_0})}(v_p),0
\right)\equiv F_{(p,t_0)}(v_p).
\end{array}\end{equation}
Now, consider also
the natural product Riemannian metric $\tilde g = g\oplus dt^2$ on
$M\times \R$, and denote by $\tilde D/dt$ the corresponding
covariant derivative.

\begin{proposition}\label{ptilde}
A curve $\tilde\gamma(t)=(\gamma(t),\tau(t))$  on $\m\times\R$
 solves
\[
\hspace*{3.4cm}\frac{\tilde D\dot{\tilde\gamma}}{dt}(t)\ =\ \tilde
F_{\tilde\gamma(t)}\, \dot{\tilde \gamma}(t) + \tilde
X_{\tilde\gamma(t)},\hspace*{3.4cm}\mathrm{({\tilde E}_0)}
\]
if and only if $\gamma$ solves $\mathrm{(E)}$ and $\tau(t)=a t+b$ for some
$a$, $b \in\R$.

Therefore, if the inextensible solutions $\tilde \gamma$ of
$\mathrm{({\tilde E}_0)}$ are complete, then so are the
trajectories for the time--dependent equation $\mathrm{(E)}$.
\end{proposition}

\begin{proof} The first assertion follows from
 $\frac{\tilde D\dot{\tilde\gamma}}{dt}=(\frac{D\dot{\gamma}}{dt},\ddot \tau)$
and formulae \eqref{def1}. For the last one, if $\gamma$ is any
inextensible solution of equation {\rm (E)}, the corresponding
curve $\tilde\gamma(t)=(\gamma(t),t)$ is an inextensible solution
of $(\tilde{E}_0)$; by assumption, $\tilde\gamma$ is complete and
so is $\gamma$.
\end{proof}

Now, we are in position to prove Theorem \ref{A1}. 

\smallskip

\begin{proof}[Proof of Theorem \ref{A1}]
If a trajectory $\gamma:[0,b)\rightarrow M$ were forward
inextensible with $0< b<+\infty$, then it should lie in a region
such as $M\times [-T,T]$ with $b < T$. By  Proposition
\ref{ptilde}, a contradiction will follow by proving the
extendability to $b$ of the trajectory $\tilde
\gamma(t)=(\gamma(t),t)$ for $\tilde M, \tilde g, \tilde X,\tilde
F$. Recall that in the region $M\times [-T,T]$, the $g$--bounds
\eqref{bx2} for $X$ and \eqref{bf} for $S$ are equal to their
counterparts \eqref{bx} for $\tilde X$ and the boundedness of the
self--adjoint part of $\tilde F$ with respect to $\tilde g$
(recall that the $\tilde g$--distance can be easily bounded in
terms of $g$ and $dt^2$, and the distance on the $dt^2$ side is
bounded by $2T$). Thus, the forward extendability of $\tilde
\gamma$ follows from Theorem \ref{A}, as required. Similar
arguments apply for proving the backward completeness case.
\end{proof}

\subsection{Trajectories under a non--autonomous potential}
\label{four}

 Throughout this subsection,  the non--autonomous problem
$\mathrm{(E^*)}$ is  studied.

\begin{remark} {\rm As a difference with Theorem \ref{A1}, now the main result
(Theorem \ref{A02}) will not be reduced directly  to the
autonomous case. The reason is that if we consider
$\tilde M, \tilde g, \tilde F$ as in the previous subsection, and
put $\tilde V: M\times \R\rightarrow \R$ simply equal to $V$, then
$\tilde \nabla \tilde V= (\nabla^MV,\partial V/\partial t)$. That
is, the component $\partial V/\partial t$ makes intrinsically
different the trajectories for $\tilde\nabla \tilde V$ and
$\nabla^M V$ (an analogous to Proposition \ref{ptilde} does not
hold). Instead, we will modify directly the proof of
Theorem \ref{A01}.}\end{remark}

\smallskip

\begin{proof}[Proof of Theorem \ref{A02}] Let $\gamma:[0,b)\rightarrow M$
be a forward inextensible solution with $0 < b<+\infty$ included
in some region $M\times [-T,T]$ with $b < T$, and let $N_T$ be as
in \eqref{bf} and $A_T, C_T$ be the constants determined by the
allowed  growth of $U=\partial V/\partial t$ in \eqref{quadratic}.
Thus, considering the steps of the proof of Theorem \ref{A01},
from \eqref{uno} and \eqref{tre} we have
\begin{eqnarray}
\frac{d}{ds}\Big(\frac{1}{2}\ u(s) + V(\gamma(s),s)\Big) \  &\leq&\
N_T u(s) + \frac{\partial V}{\partial s}(\gamma(s),s)\nonumber\\
&\le&\ N_T u(s) + A_T d^2(\gamma(s),p_0) + C_T\nonumber\\
&\le& \ N_T u(s) + A_T l_\gamma^2(s) + C_T.\label{dis1mod}
\end{eqnarray}
As
\[
\int_0^t l_\gamma^2(s) ds\ \le \ T l_\gamma^2(t)\quad \hbox{for all $t \in [0,b)$,}
\]
integrating \eqref{dis1mod} we have that equation \eqref{eq1} changes into:
\begin{equation} \label{eq1mod}
\dot v(t) - 2 N_T v(t)\, \le\ 2 (\alpha_\gamma -
V(\gamma(t),t)) + 2 T A_T l_\gamma^2(t)
\end{equation}
 for all $ t\in[0,b)$, with $\alpha_\gamma = \frac12 u(0) + V(\gamma(0),0) + T C_T$.
Taking into account also the at most quadratic bound for $-V$, we
can choose constants $A$, $C > 0$ and construct the
function $\V (s) = -A s^2 -C$ so that \eqref{eq1mod} yields:
\begin{equation}\label{eq1mod2}
\dot v(t) - 2 N_T v(t)\, <\ 2 (\alpha - \V(l_\gamma(t)))
\qquad \text{for all}\quad t\in[0,b),
\end{equation}
where  $\alpha > $ max$\{\alpha_\gamma, \V(0)\}$. Equation
\eqref{eq1mod2} is formally equal to \eqref{eq2} in Theorem
\ref{A01}. So, reasoning as in the proof of Theorem \ref{A01}
we show the extendability
of $\gamma$ through $b$, a contradiction.

Vice versa, if $\gamma :(-b,0] \to \m$ is a backward inextendible solution
with $0<b<+\infty$, we can consider $T>b$ and $\tilde \gamma(t) := \gamma(-t)$ in $[0,b)$
and, as from the lower boundedness of $S$ and the quadratic growth of $-\frac{\partial V}{\partial t}$
along finite times it follows
\[
\begin{split}
\frac{d}{ds}\Big(\frac{1}{2}\ u(-s) + V(\gamma(-s),-s)\Big) \ &=\
- g\left(S_{(\gamma(-s),-s)} \dot\gamma(-s),\dot\gamma(-s)\right)\\
&\quad - \frac{\partial V}{\partial s}(\gamma(-s),-s)\\
&\le\ N_T u(-s) + A_T d^2(\gamma(-s),p_0) + C_T,
\end{split}
\]
we repeat the proof with \eqref{eq1mod} stated for $\tilde\gamma(t)$.
\end{proof}

\begin{remark} \label{rgordon} The maximum allowed  growth permitted for $\partial
V/\partial t$ in Theorem \ref{A02} is both, very general and
consequent with our other hypotheses. However, checking the
proofs, other possible bounds could be taken into account. In
order to discuss the accuracy of our bound for $\partial
V/\partial t$, next: (a) we will compare
Theorem \ref{A02} with the consequences of Theorem \ref{A1} for
potentials, and (b) we introduce an alternative bound on
$\partial V/\partial t$ applicable when $V$ is lower bounded along
finite times. This suggests that, even though the optimality of
all the other bounds have been carefully checked previously, the
optimal bounds for $\partial V/\partial t$ can be studied further.
\end{remark}

According to the claim (a) in Remark \ref{rgordon}, the
application of Theorem \ref{A1} for the case of potentials yields:

\begin{corollary}\label{A10}
Let $(\m,g)$ be a complete Riemannian manifold, consider a $(1,1)$
tensor field $F$, eventually time--dependent, with self--adjoint
component $S$, and let $V:\m\times\R \rightarrow \R$ be a smooth
potential. If $S$ is bounded along finite times and $\nabla^\m
V(p,t)$ grows at most linearly in $\m$ along finite times, then
each inextensible solution of $\mathrm{(E^*)}$ must be complete.
\end{corollary}

\begin{proof}[Relation between Corollary \ref{A10} and Theorem \ref{A02}]
Choose $p_0\in M$ and, by using the completeness of $g$, take a
starshaped domain $\mathcal{D}\subset T_{p_0}M$ so that the
exponential map $\exp_{p_0}: \mathcal{D}\rightarrow M$ is a
diffeomorphism onto its image $\exp_{p_0}(\mathcal{D})$, and this
image is dense in $M$. Under the hypotheses of Corollary
\ref{A10}, let $A, C:[0,+\infty)\rightarrow [0,+\infty)$ be
strictly increasing $C^1$ functions which satisfy $A(n)>A_{n+1}$,
$C(n)>C_{n+1}$ where $A_{n+1}, C_{n+1}$ are the constants $A_T,
C_T$ obtained in \eqref{bx2} (from the at most linear growth of
$\nabla^\m V$) for $T=n+1$, and $n$ is any nonnegative integer.
Now, for each $p\in\exp_{p_0}(\mathcal{D})$ let
$\gamma_p:[0,l_p]\rightarrow \exp_{p_0}(\mathcal{D})$ be the
unique unit geodesic from $p_0$ to $p$. Clearly,
\[
\begin{split}
&-(V(p,l_p) - V(p_0,0))\ =\ -\int_0^{l_p}\frac{d V}{d s}(\gamma_p(s),s) ds\\
&\quad\qquad =\ -\int_0^{l_p}\big(g(\nabla^MV(\gamma_p(s),s),\dot\gamma_p(s))+\frac{\partial
V}{\partial s}(\gamma_p(s),s)\big) ds \\
&\quad\qquad \leq\ \int_0^{l_p} |\nabla^M V(\gamma_p(s),s)| ds -
\int_0^{l_p} \frac{\partial V}{\partial s}(\gamma_p(s),s) ds\\
&\quad\qquad \leq\ A(l_p) l_p^2 + C(l_p) -\int_0^{l_p} \frac{\partial
V}{\partial s}(\gamma_p(s),s) ds.
\end{split}
\]
Taking into account that $l_p=d(p_0,p)$ and the density of
$\exp_{p_0}(\mathcal{D})$ we have: (i) in the autonomous case
($\frac{\partial V}{\partial s}\equiv 0$), $-V$ grows at most
quadratically, that is, Corollary \ref{A10} is a particular case
of Theorem \ref{A02}, (ii) in the non--autonomous case, both
results are independent: Corollary \ref{A10} does not require any
bound for $\frac{\partial V}{\partial s}$, and the bound required
by Theorem \ref{A02} is independent of the relation between the
bounds for $-V$ and $|\nabla^MV|$.
\end{proof}

According to the claim (b) in Remark \ref{rgordon}, consider the
following result\footnote{This result can be proven easily by
using the previous techniques. At any case, the full details will
be written in the proceedings paper \cite{CRS}, where the results
of the present article were announced.}:

\begin{proposition}\label{G01} Let $(\m,g)$ be a complete Riemannian manifold, $F$ a
smooth time--independent $(1,1)$ tensor field with self--adjoint
component $S$ and $V:\m\times\R \rightarrow \R$ a smooth
potential. Assume that $\|S\|$ is bounded, and there exist
continuous functions
$\alpha_0, \beta_0: \R\rightarrow \R$, $\alpha_0, \beta_0 >0$
such that $V(p,t)>\beta_0(t)$ (i.e., $V$ is bounded from below along finite times)
and:
\[
\left|\frac{\partial V}{\partial t}(p,t)\right|\ \le\ \alpha_0(t)
(V(p,t) - \beta_0(t))\quad \hbox{for all \, $(p,t)\in \m\times
\R$.}
\]
Then, each inextensible solution of equation $\mathrm{(E^*)}$ must be
complete.
\end{proposition}

\begin{proof}[Relation between Proposition \ref{G01} and Theorem \ref{A02}]
Proposition \ref{G01} impos\-es a strong re\-stric\-tion which was not
present in Theorem \ref{A02}, namely, the boundedness from below
(along finite times) of $V$. However, once this assumption is
admitted, Proposition \ref{G01} shows two remarkable properties:

(i) When $V$ grows fast towards infinity (say, in a superquadratic
way) such a fast growth is also permitted for $|\frac{\partial
V}{\partial t}|$, and

(ii) The lower bound for $V$ makes natural the assumption {\sl
``$V$ is a proper function on $\m\times\R$''} (i.e., the inverse
image $V^{-1}(K)$ is compact in $\m\times\R$ for any compact
subset $K$ in $\R$). Under this assumption, the completeness of
$g$ can be removed (recall footnote \ref{foot} or see \cite{CRS}
for details).
\end{proof}

As commented in the Introduction, the importance of the
completeness theorems stated in the present paper leans not only
upon themselves but also upon their applications in other fields,
for example in Lorentzian Geometry. In fact, we consider the
following application of the case of non--autonomous potentials to
an important class of spacetimes in General Relativity.

\begin{example}\label{eppwave} {\em Application to the geodesic
completeness of pp--waves }. The so--called {\em parallely
propagated waves}, or just {\em pp--waves}, are the relativistic
spacetimes on $\R^4$ endowed with the Lorentzian metric
\[
ds^2 \ =\ dx^2+dy^2 + 2dudv + H(x,y,u) du^2,
\]
where $(x,y,u,v)$ are the natural coordinates of $\R^4$.
It is known that the
geodesic completeness of these spacetimes is equivalent to the
completeness of the trajectories $\gamma(u)=(x(u),y(u))$ for the
(purely Riemannian) non--autonomous problem on $\R^2$:
\[
\ddot \gamma (u)\ =\ \frac12 \nabla^{\R^2}H(\gamma(u),u),
\]
where $V(x,y,u):= - H(x,y,u)/2$ plays the role of a time--dependent
potential with time coordinate $u$ (see \cite[Theorem 3.2]{CFS}
for more details). Therefore, all the results of Subsection
\ref{four} provide criteria which ensure the completeness of this
type of spacetimes.

In the particular case of the so--called {\em  plane waves} the
expression of $H$ is quadratic in $x, y$, i.e.:
\begin{equation} \label{ehpw}
H(x,y,u)\ =\ f_{11}(u) x^2 - f_{22}(u) y^2 + 2 f_{12}(u) xy,
\end{equation}
for some $C^2$--functions $f_{11}$, $f_{22}$, $f_{12}$. The completeness of plane
waves was known  because a direct integration of the geodesics is
possible (see for example \cite[Chapter 13]{BEE}). However, it can
be deduced easily from our results (recall that both Theorems
\ref{A02} and \ref{A1} are applicable, as $V=-H/2$ grows at most
quadratically along finite times and its $\nabla^{\R^2}$--gradient
grows at most linearly). This property is important because our
results also ensure: \begin{quote} {\em Any  pp--wave such that its
function $H$ behaves qualitatively as \eqref{ehpw} is geo\-de\-si\-cally
complete.} \end{quote} In fact, as claimed in \cite{FS_CQG},
physically realistic pp--waves must have a function  $|H|$ with a
growth at most quadratic along finite times (being the quadratic
case a limit case of the properly realistic subquadratic case).
Then, as an
interpretation of our result: 
{\em no physically realistic pp--wave develops singularities.} 
Such a property goes in the same direction that other
geometric properties on causality and boundaries for pp--waves,
developed in \cite{FS_CQG}, \cite{FS_JHEP}.
\end{example}

 \section*{Acknowledgement}
This paper is partially
supported by the {\sl Azione Integrata Italia--Spagna}
HI2008--0106. Furthermore, A.M. Candela has been partially supported by M.I.U.R. (research funds ex 40\% and
60\%) while A. Romero and M. S\'anchez are partially supported by Spanish
Grants with FEDER funds  MTM2010--18099 (MICINN) and P09--FQM--4496
(J. de Andaluc\'{\i}a).

\vspace{12mm}

\noindent
Anna Maria Candela\\
Dipartimento di Matematica, \\ Universit\`a degli Studi di Bari ``Aldo Moro'',\\
Via E. Orabona 4, 70125 Bari, Italy\\
e-mail: candela@dm.uniba.it

\vspace{9mm}

\noindent
Alfonso Romero$^\dagger$ and Miguel S\'anchez$^\ddagger$\\
Departamento de Geometr\'{\i}a y Topolog\'{\i}a,\\
Facultad de Ciencias, 
Universidad de Granada,\\
Avenida Fuentenueva s/n,\\
18071 Granada, Spain\\
e-mail addresses: $^\dagger$aromero@ugr.es,
$^\ddagger$sanchezm@ugr.es

\end{document}